\documentclass[11pt]{article}
\usepackage[margin=1in]{geometry} 
\usepackage{amssymb,amsfonts,amsmath,bbm,mathrsfs,stmaryrd}
\usepackage{mathabx}
\usepackage{xcolor}
\usepackage{url}

\usepackage{enumerate}

\usepackage[colorlinks,
             linkcolor=black!75!red,
             citecolor=blue,
             pdftitle={},
             pdfproducer={pdfLaTeX},
             pdfpagemode=None,
             bookmarksopen=true
             bookmarksnumbered=true]{hyperref}

\usepackage{tikz}
\usetikzlibrary{arrows,calc,decorations.pathreplacing,decorations.markings,intersections,shapes.geometric,through,fit,shapes.symbols,positioning,decorations.pathmorphing}

\usepackage{braket}

\usepackage[amsmath,thmmarks,hyperref]{ntheorem}
\usepackage{cleveref}

\creflabelformat{enumi}{#2(#1)#3}

\crefname{section}{Section}{Sections}
\crefformat{section}{#2Section~#1#3} 
\Crefformat{section}{#2Section~#1#3} 

\crefname{subsection}{\S}{\S\S}
\AtBeginDocument{%
  \crefformat{subsection}{#2\S#1#3}%
  \Crefformat{subsection}{#2\S#1#3}% 
}

\crefname{subsubsection}{\S}{\S\S}
\AtBeginDocument{%
  \crefformat{subsubsection}{#2\S#1#3}%
  \Crefformat{subsubsection}{#2\S#1#3}% 
}

\theoremstyle{plain}

\newtheorem{lemma}{Lemma}[section]
\newtheorem{proposition}[lemma]{Proposition}

\newtheorem{theorem}[lemma]{Theorem}

\theoremstyle{nonumberplain}

\theoremstyle{plain}
\theorembodyfont{\upshape}
\theoremsymbol{\ensuremath{\blacklozenge}}

\newtheorem{definition}[lemma]{Definition}
\newtheorem{example}[lemma]{Example}
\newtheorem{remark}[lemma]{Remark}

\crefname{definition}{definition}{definitions}
\crefformat{definition}{#2definition~#1#3} 
\Crefformat{definition}{#2Definition~#1#3} 

\crefname{ex}{example}{examples}
\crefformat{example}{#2example~#1#3} 
\Crefformat{example}{#2Example~#1#3} 

\crefname{remark}{remark}{remarks}
\crefformat{remark}{#2remark~#1#3} 
\Crefformat{remark}{#2Remark~#1#3} 

\crefname{convention}{convention}{conventions}
\crefformat{convention}{#2convention~#1#3} 
\Crefformat{convention}{#2Convention~#1#3} 

\crefname{notation}{notation}{notations}
\crefformat{notation}{#2notation~#1#3} 
\Crefformat{notation}{#2Notation~#1#3} 

\crefname{table}{table}{tables}
\crefformat{table}{#2table~#1#3} 
\Crefformat{table}{#2Table~#1#3}

\crefname{lemma}{lemma}{lemmas}
\crefformat{lemma}{#2lemma~#1#3} 
\Crefformat{lemma}{#2Lemma~#1#3} 

\crefname{proposition}{proposition}{propositions}
\crefformat{proposition}{#2proposition~#1#3} 
\Crefformat{proposition}{#2Proposition~#1#3} 

\crefname{corollary}{corollary}{corollaries}
\crefformat{corollary}{#2corollary~#1#3} 
\Crefformat{corollary}{#2Corollary~#1#3} 

\crefname{theorem}{theorem}{theorems}
\crefformat{theorem}{#2theorem~#1#3} 
\Crefformat{theorem}{#2Theorem~#1#3} 

\crefname{enumi}{}{}
\crefformat{enumi}{(#2#1#3)}
\Crefformat{enumi}{(#2#1#3)}

\crefname{assumption}{assumption}{Assumptions}
\crefformat{assumption}{#2assumption~#1#3} 
\Crefformat{assumption}{#2Assumption~#1#3} 

\crefname{equation}{}{}
\crefformat{equation}{(#2#1#3)} 
\Crefformat{equation}{(#2#1#3)}

\numberwithin{equation}{section}

\theoremstyle{nonumberplain}
\theoremsymbol{\ensuremath{\blacksquare}}

\newtheorem{proof}{Proof}

\newcommand\bC{{\mathbb C}}

\newcommand\bZ{{\mathbb Z}}

\DeclareMathOperator{\id}{id}

\newcommand{\co}[1]{C_0\left({#1}\right)}

\newcommand{\couh}[1]{C_0^u\left(\widehat{{#1}}\right)}

\title{Fields of locally compact quantum groups: continuity and pushouts}
\author{Alexandru Chirvasitu}

\begin{document}

\date{}

\newcommand{\Addresses}{{% additional braces for segregating \footnotesize
  \bigskip
  \footnotesize

  \textsc{Department of Mathematics, University at Buffalo, Buffalo,
    NY 14260-2900, USA}\par\nopagebreak \textit{E-mail address}:
  \texttt{achirvas@buffalo.edu}

% %   \medskip
% %   
% %   \textsc{Department of Mathematics, institution,
% %     address}\par\nopagebreak \textit{E-mail address}:
% %   \texttt{??}
% % 
}}

\maketitle

\begin{abstract}
  We prove that (a) discrete compact quantum groups (or more generally locally compact, under additional hypotheses) with coamenable dual are continuous fields over their central closed quantum subgroups, and (b) the same holds for free products of discrete quantum groups with coamenable dual amalgamated over a common central subgroup. Along the way we also show that free products of continuous fields of $C^*$-algebras are again free via a Fell-topology characterization for $C^*$-field continuity, recovering a result of Blanchard's in a somewhat more general setting.
\end{abstract}


\vspace{.5cm}


%\tableofcontents

%%%%%%%%%%%%%%%%%%%%%%%%%%%%%%%%%%%%%%%%%%%%%%%%%%%%%%%%%%%%%%%%%%%%%%%%%%%%%%%%%%%%%%%%%%%%%%%%%%%%%%%%%%%%%%%%%%
%%%%%%%%%%%%%%%%%%%%%%%%%%%%%%%%%%%%%%%%%%%%%%%%%%%%%%%%%%%%%%%%%%%%%%%%%%%%%%%%%%%%%%%%%%%%%%%%%%%%%%%%%%%%%%%%%%
\section*{Introduction}

The initial motivation for the present note was the desire to extend one of the main results of \cite{chi-rf} (Theorem 3.2 therein) in two ways: from plain (``classical'') to {\it quantum} groups, and from discrete to locally compact. The result appears below as \Cref{th:1gp}:

\begin{theorem}\label{th:1gp-intro}
  For
  \begin{itemize}
  \item a locally compact quantum group $G$ with coamenable dual
  \item with a central closed quantum subgroup $H\le G$
  \item such that $G/H$ has coamenable dual (automatic if $G$ is discrete)
  \end{itemize}
  the group $C^*$-algebra $C_0^u(\widehat{G})$ forms a continuous field over the group algebra $C_0^u(\widehat{H})$ of any central closed quantum subgroup $H\le G$.
\end{theorem}

We recall the terminology and notation below in \Cref{subse:lcqg} (e.g. $\widehat{G}$ for the {\it dual} $\widehat{G}$ of $G$), pausing here only to remind the reader that a {\it reduced locally compact quantum group} $G$ in the sense of \cite{kv-lcqg} (see also \cite[Definition 8.1.17]{Ti08}) consists of a generally non-unital $C^*$-algebra $C^r_0(G)$ (thought of as a algebra of continuous functions vanishing at infinity on $G$) equipped with
\begin{itemize}
\item a coassociative {\it comultiplication} morphism $\Delta:C^r_0(G)\to C_0^r(G)^{\otimes 2}$ (minimal $C^*$ tensor product) in the sense of \Cref{def:mor} (i.e. landing in the multiplier algebra);
\item left and right-invariant {\it Haar weights} (on which we do not elaborate);
\item and hence a von Neumann algebra $L^{\infty}(G)$ ( \cite[Definition 8.1.4]{Ti08} or \cite{kv-vn}) attached to one of these invariant weights via the GNS construction.
\end{itemize}
 
Since furthermore \cite[Theorem 3.2]{chi-rf} handles {\it pushouts} of amenable groups $G_i$, $i=1,2$ over a common central subgroup $H$, it seemed desirable to have an analogous extension here (\Cref{th:main}):

\begin{theorem}\label{th:main-intro}
    Let $G_i$, $i\in I$ be a family of discrete quantum groups with coamenable duals and a common central closed quantum subgroup $H\le G_i$. Then, the $C^*$ pushout
  \begin{equation*}
    \Asterisk_{C^u\left(\widehat{H}\right)} C^u\left(\widehat{G_i}\right)
  \end{equation*}
  is a continuous field over the commutative $C^*$-algebra $C^u\left(\widehat{H}\right)$.
\end{theorem}

Note that as opposed to \Cref{th:1gp-intro}, where $G$ is locally compact, here the $G_i$ are {\it discrete} (equivalently, $C_0^u(G)$ are unital, hence the missing `$0$' superscript in $C^u_0$): this is to avoid the unpleasantness of working with non-unital pushouts.

One natural path to \Cref{th:main-intro} (or something like it, perhaps covering pushouts of only {\it two} quantum groups) would be to start with \Cref{th:1gp-intro} and apply \cite[Theorem 3.7]{bl-fr}, to the effect that a pushout $A*_CB$ of fields of $C^*$-algebras continuous over a central $C^*$-algebra $C$ is again continuous over $C$. This was the initial intention, but in the process of unwinding that cited result the proof appeared to contain a gap. For that reason, it seemed worthwhile to try to recover \cite[Theorem 3.7]{bl-fr} here via a different approach (\Cref{th:psh-cont}):

\begin{theorem}\label{th:psh-cont-intro}
  Let $X$ be a compact Hausdorff space and $A_i$, $i\in I$ a family of unital $C(X)$-algebras. If all $A_i$ are continuous then so is the pushout
  \begin{equation*}
    A:=\Asterisk_{C(X)} A_i.
  \end{equation*}  
\end{theorem}

%%%%%%%%%%%%%%%%%%%%%%%%%%%%%%%%%%%%%%%%%%%%%%%%%%%%%%%%%%%%%%%%%%%%%%%%%%%%%%%%%%%%%%%%%%%%%%%%%%%%%%%%%%%%%%%%%%
\subsection*{Acknowledgements}

This work is partially funded by NSF grant DMS-2001128.

I am grateful for numerous enlightening conversations with Amaury Freslon, Mike Brannan and Ami Viselter on and around the topics covered here.

%%%%%%%%%%%%%%%%%%%%%%%%%%%%%%%%%%%%%%%%%%%%%%%%%%%%%%%%%%%%%%%%%%%%%%%%%%%%%%%%%%%%%%%%%%%%%%%%%%%%%%%%%%%%%%%%%%
%%%%%%%%%%%%%%%%%%%%%%%%%%%%%%%%%%%%%%%%%%%%%%%%%%%%%%%%%%%%%%%%%%%%%%%%%%%%%%%%%%%%%%%%%%%%%%%%%%%%%%%%%%%%%%%%%%
\section{Preliminaries}\label{se.prel}

$C^*$-algebras are not assumed unital unless we do so explicitly, and morphisms are defined as is customary for generally-non-unital $C^*$-algebras (e.g. \cite[Notations and conventions]{kv-lcqg}):

\begin{definition}\label{def:mor}
  Let $A$ and $B$ be two possibly-non-unital $C^*$-algebras. A {\it morphism} $A\to B$ is a linear, bounded, multiplicative and $*$-preserving map $f:A\to M(B)$ (the {\it multiplier algebra} of $B$; \cite[\S 2.2]{wo} or \cite[\S II.7.3]{blk-oa}) that is {\it non-degenerate} in the sense that
  \begin{equation*}
    f(A)B:=\text{span}\{f(a)b\ |\ a\in A,\ b\in B\}
  \end{equation*}
  is dense in $B$. 
\end{definition}

As noted in \cite[Notations and conventions]{kv-lcqg}, morphisms $A\to B$ in this sense extend uniquely to unital morphisms $M(A)\to M(B)$ that are {\it strictly continuous} on bounded subsets. Recall that strict continuity means continuity with respect to the seminorms
\begin{equation*}
  M(A)\ni x\mapsto \|xa\| + \|ax\|,\ a\in A. 
\end{equation*}

%%%%%%%%%%%%%%%%%%%%%%%%%%%%%%%%%%%%%%%%%%%%%%%%%%%%%%%%%%%%%%%%%%%%%%%%%%%%%
\subsection{Locally compact quantum groups}\label{subse:lcqg}

We will need some background on these (abbreviated as LCQGs) as introduced in \cite{kv-lcqg}. Additional sources include the excellent textbook \cite{Ti08} as well as various other papers cited in the process of (very briefly) recalling some of the relevant notions.

In addition to the structure reviewed briefly in the introduction, on can define, for an LCQG $G$,

\begin{itemize}
\item the {\it universal} version $C_0^u(G)$ \cite[\S 11]{kus-univ} that is again a $C^*$-algebra equipped with a coassociative morphism $\Delta:C^u_0(G)\to C_0^u(G)^{\otimes 2}$ and a surjection $C_0^u(G)\to C_0^r(G)$ intertwining the comultiplications;
\item the {\it Pontryagin dual} $\widehat{G}$ of $G$ (\cite[Definition 8.3.14]{Ti08}), whose underlying universal $C^*$-algebra $C_0^u(\widehat{G})$ analogizes the universal group algebra of $G$ (in particular, {\it representations} of $G$ on Hilbert spaces are precisely representations of $C_0^u(\widehat{G})$ as a $C^*$-algebra \cite[\S 5]{kus-univ}). 
\end{itemize}

The following version of the notion of discreteness will be most directly applicable below (see also \cite[\S 3.3]{Ti08} or \cite{vd-disc}).

\begin{definition}
  $G$ is {\it discrete} if $C_0^r(\widehat{G})$ is unital.
\end{definition}

Regarding the relationship between $C_0^u$ and $C_0^r$, recall \cite[Definition 3.1, Theorem 3.1]{bt-lcqg}:

\begin{definition}\label{def:coam}
  A locally compact quantum group $G$ is {\it coamenable} if either of the two following equivalent conditions holds:
  \begin{itemize}
  \item there is a character $\varepsilon:C_0^r(G)\to \bC$ such that $(\id\otimes\varepsilon)\Delta=\id_{C_0^r(G)}$;
  \item the surjection $C_0^u(G)\to C_0^r(G)$ is an isomorphism. 
  \end{itemize}
\end{definition}

As in the classical setting, we can talk about closed quantum subgroups (\cite[Definition 2.5]{vaes-imp} and \cite[Definitions 3.1 and 3.2]{dkss}):

\begin{definition}
  Let $G$ be a locally compact quantum group. 
  \begin{enumerate}[(a)]
  \item A {\it (Vaes-)closed quantum subgroup} $H\le G$ is a locally compact quantum group $H$ equipped with a normal embedding
    \begin{equation*}
      L^{\infty}(\widehat{H})\to L^{\infty}(\widehat{G})
    \end{equation*}
    intertwining the comultiplications. 
  \item This then induces a surjection $C_0^u(G)\to C_0^u(H)$, thus realizing $H$ as a {\it Woronowicz-closed quantum subgroup of $G$} (\cite[Definition 3.2, Theorems 3.5 and 3.6]{dkss}).
  \item The (Vaes-)closed $H\le G$ is {\it central} \cite[\S 1.1]{kss-cent} if
    \begin{equation*}
      L^{\infty}(\widehat{H})\subseteq L^{\infty}(\widehat{G})
    \end{equation*}
    is contained in the center. 
  \end{enumerate}
\end{definition}

%%%%%%%%%%%%%%%%%%%%%%%%%%%%%%%%%%%%%%%%%%%%%%%%%%%%%%%%%%%%%%%%%%%%%%%%%%%%%
\subsection{Fields of $C^*$-algebras}

Denoting, as usual, by $\co{X}$ the algebra of continuous functions vanishing at infinity on a locally compact Hausdorff space $X$, we work with $\co{X}$-algebras in the sense of \cite[Introduction]{bl-was} or \cite[Definition 2.1]{bl-k} (see also \cite[Definition 1.1]{bl-tens} and \cite[Definition 2.2]{bl-fr} for the case of compact $X$):

\begin{definition}
  A {\it $\co{X}$-algebra} is a (possibly non-unital) $C^*$-algebra $A$ equipped with a non-degenerate morphism from $\co{X}$ to the center of the multiplier algebra $M(A)$. 
\end{definition}

One can form, for every point $x\in X$, the {\it fiber} $A_x$ of $A$ at $x$ as
\begin{equation*}
  A_x = A/\{\text{ideal generated by $ma$ and $am$}\}
\end{equation*}
where $a\in A$ and $m\in M(A)$ ranges over the image through $\co{X}\to M(A)$ of the ideal of functions vanishing at $x$.

For $a\in A$ we follow \cite{bl-fr} in denoting by $a_x$ the image of $a$ through $A\to A_x$, and when the need arises to distinguish between the norms of the various $A_x$ we write $\|\cdot\|_x$ for the latter. Recall \cite[D\'efinition 3.1]{blnch} (see also \cite[Definition 2.2]{bl-k} or \cite[discussion following Definition 2.2]{bl-fr}):

\begin{definition}\label{def:contf}
  The $\co{X}$-algebra $A$ is {\it continuous} as such, or {\it a continuous field} over $X$ if
  \begin{equation}\label{eq:iscont}
    X\ni x\mapsto \|a_x\|
  \end{equation}
  is continuous for every $a\in A$.  
\end{definition}

\begin{remark}\label{re:low}
  Note that \Cref{eq:iscont} is always {\it upper} semicontinuous \cite[Proposition 1.2]{rieff-flds}, so it is lower semicontinuity that is the core issue motivating \Cref{def:contf}.
\end{remark}

When working with {\it pushouts} (e.g. \cite{bl-fr,ped-psh}) we specialize to the unital case: $A_i$ will typically be unital algebras equipped with unital morphisms $C(X)\to A_i$, thus allowing the formation of the pushout
\begin{equation}\label{eq:psh}
  \Asterisk_{C(X)} A_i.
\end{equation}

To make sense of \Cref{le:fell} below we need some background on the {\it Fell topology} on (unitary isomorphism classes of) representations and on {\it weak containment}. We refer to \cite{fell-wk1,fell-smth}, \cite[Chapter 3]{dix-cast} and \cite[Appendix F]{bhv} for material on these topics.

In the statement of \Cref{le:fell}, for a point $x\in X$
\begin{equation*}
  p_x:\co{X}\to \bC
\end{equation*}
denotes the evaluation character, regarded as a $\co{X}$-representation.

\begin{lemma}\label{le:fell}
  A $\co{X}$-algebra $A$ is a continuous field if and only if for every convergent net
  \begin{equation}\label{eq:net}
    x_{\alpha}\to x\in X
  \end{equation}
  every representation of $A$ that factors through $A\to A_x$ is a Fell limit a net of representations that factor through $A\to A_{x_{\alpha}}$.
\end{lemma}
\begin{proof}
  We prove the two implications separately.

  {\bf $(\Rightarrow)$} Suppose $A$ is continuous over $X$. To prove the desired conclusion we have to show (\cite[Theorem 3.4.4]{dix-cast}) that for a net \Cref{eq:net} and a representation
  \begin{equation}\label{eq:rho}
    \rho:A\to A_x\to B(H)
  \end{equation}
  where $\co{X}$ acts via $p:=p_x$ we can find representations
  \begin{equation*}
    \rho_{\alpha}:A\to B(H_{\alpha})
  \end{equation*}
  such that
  \begin{itemize}
  \item $\rho_{\alpha}$ restricts to $\co{X}$ as the character $p_{\alpha}$, and
  \item the intersection of the kernels of $\rho_{\alpha}$ is contained in $\ker \rho$. 
  \end{itemize}
Now choose a net of representations
  \begin{equation}\label{eq:rhoa}
    \rho_{\alpha}:A\to A_{x_{\alpha}} \to B(H_{\alpha}),
  \end{equation}
  faithful on $A_{x_{\alpha}}$ respectively. By the continuity of the field $A$ over $X$, any element annihilated by all $\rho_{\alpha}$ must also be annihilated by $A\to A_x$, and hence by \Cref{eq:rho}. 

  {\bf $(\Leftarrow)$} Conversely, to prove that $A$ is continuous over $X$ we have to argue that for every convergent net \Cref{eq:net} and every $a\in A$ we have
  \begin{equation}\label{eq:lmsp}
    \|a_x\|\le \limsup_{\alpha}\|a_{x_{\alpha}}\|. 
  \end{equation}
  The hypothesis says that we can find a net of representations \Cref{eq:rhoa} Fell-converging to some representation \Cref{eq:rhoa} faithful on $A_x$. Since the Fell topology does not distinguish between sums of copies of a given representation (regardless of the cardinality of the set of summands), we may as well assume that
  \begin{equation*}
    \rho\cong \rho^{\oplus\aleph_0}
  \end{equation*}
  and similarly for all $\rho_{\alpha}$. But in that case \cite[Lemma 2.4]{fell-wk1} shows that
  \begin{itemize}
  \item we can realize $\rho$ concretely and non-degenerately on some large Hilbert space $H$
  \item which also houses (possibly degenerate) copies of $\rho_{\alpha}$
  \item so that for every $a\in A$ and every $\xi\in H$ we have
    \begin{equation*}
      \lim_{\alpha}\|\rho(a)\xi-\rho_{\alpha}(a)\xi\|=0. 
    \end{equation*}
  \end{itemize}
  Since $\rho$ is assumed faithful on $A_x$, \Cref{eq:lmsp} follows.
\end{proof}

%%%%%%%%%%%%%%%%%%%%%%%%%%%%%%%%%%%%%%%%%%%%%%%%%%%%%%%%%%%%%%%%%%%%%%%%%%%%%%%%%%%%%%%%%%%%%%%%%%%%%%%%%%%%%%%%%%
%%%%%%%%%%%%%%%%%%%%%%%%%%%%%%%%%%%%%%%%%%%%%%%%%%%%%%%%%%%%%%%%%%%%%%%%%%%%%%%%%%%%%%%%%%%%%%%%%%%%%%%%%%%%%%%%%%
\section{Fields over central locally compact quantum groups}\label{se:lcqg}

We begin with

\begin{theorem}\label{th:1gp}
  Let $G$ be an LCQG with coamenable dual and $H\le G$ a central closed quantum subgroup.
  \begin{enumerate}[(a)]
  \item\label{item:1} If $G/H$ also has coamenable dual then $\couh{H}\to \couh{G}$ is a continuous field.
  \item\label{item:2} The hypothesis in \Cref{item:1} is automatic if $G$ is discrete.  
  \end{enumerate}
\end{theorem}
\begin{proof}
We treat the two clauses separately.
  
  {\bf Part \Cref{item:1}} We mimic the proof of \cite[Theorem 3.2]{chi-rf}. Having chosen a convergent net
  \begin{equation*}
    x_{\alpha}\to x\in \widehat{H}
  \end{equation*}
  of characters, we have to argue that an arbitrary unitary representation $\rho$ of $G$ on which $H$ acts via $p_{x}$ is in the Fell closure of a family of unitary representations where $H$ acts by $p_{x_{\alpha}}$. The coamenability condition ensures that $G/H$ has coamenable dual and hence its regular representation weakly contains its trivial representation ${\bf 1}_{G/H}$. We thus similarly have the weak containment
  \begin{equation*}
    {\bf 1}_G\preceq \mathrm{Ind}_H^G({\bf 1}_H),
  \end{equation*}
  and it follows that
  \begin{equation*}
    \rho\cong \rho\otimes {\bf 1}_G\preceq \rho\otimes \mathrm{Ind}_H^G({\bf 1}_H);
  \end{equation*}
  the latter representation is a Fell-topology limit of
  \begin{equation*}
    \rho\otimes \mathrm{Ind}_H^G(p_{x_{\alpha}x^{-1}})
  \end{equation*}
  where $H$ acts respectively by $p_{x_{\alpha}}$, hence the conclusion.

  {\bf Part \Cref{item:2}} According to \cite[Theorem 3.2]{cr-her} the amenability of $G$ entails that of $G/H$ ,whereas by \cite[Corollary 9.6]{bct-reps} amenability and dual coamenability are equivalent for discrete quantum groups.
\end{proof}

Next, \Cref{th:1gp} extends to arbitrary pushouts.

\begin{theorem}\label{th:main}
  Let $G_i$, $i\in I$ be a family of discrete quantum groups with coamenable duals and a common central closed quantum subgroup $H\le G_i$. Then, the $C^*$ pushout
  \begin{equation*}
    \Asterisk_{C^u\left(\widehat{H}\right)} C^u\left(\widehat{G_i}\right)
  \end{equation*}
  is a continuous field over the commutative $C^*$-algebra $C^u\left(\widehat{H}\right)$.
\end{theorem}
\begin{proof}
  Once we have \Cref{th:1gp} we can conclude
  \begin{itemize}
  \item for pushouts of {\it two} CQGs by \Cref{th:psh-cont} (or \cite[Theorem 3.7]{bl-fr}, but see \Cref{subse:lit} for an aside on its proof), stating that
    \begin{equation*}
      A_1*_{C(X)}A_2
    \end{equation*}
    is continuous whenever $A_i$ are;
  \item for finite families $\{G_i\}$ by induction;
  \item in general, by taking filtered colimits over the finite subsets of the index set $I$ (\Cref{th:filt}).
  \end{itemize}
  This finishes the proof, modulo the last item regarding colimits.
\end{proof}

It remains to address the filtered-colimit claim that the proof of \Cref{th:main} punts on:

\begin{proposition}\label{th:filt}
  Let $X$ be a locally compact Hausdorff space, $(I,\le)$ a filtered poset, and
  \begin{equation*}
    \iota_{ji}:A_i\to A_j,\ \forall i\le j\in I
  \end{equation*}
  a functor from $I$ to the category of $\co{X}$-algebras.

  If all $A_i$ are continuous then so is
  \begin{equation*}
    A:=\varinjlim_{i\in I} A_i
  \end{equation*}
\end{proposition}
\begin{proof}
  As noted before (\Cref{re:low}) we need lower semicontinuity, since upper semicontinuity is automatic. Concretely, having fixed an element $a\in A$, a point $x\in X$ and $\varepsilon>0$, we have to argue that for $y$ ranging over some neighborhood of $x$ we have
  \begin{equation}\label{eq:yxe}
    \|a_y\|>\|a_x\|-\varepsilon. 
  \end{equation}
  Denote by $\iota_i:A_i\to A$ the structure map of the colimit. The conclusion follows by
  \begin{itemize}
  \item approximating $a$ (and hence $a_x$ and $a_y$) arbitrarily well by elements $\iota_i(a_i)$ for $a_i\in A_i$
  \item for which $\|a_i\|$ (norm in $A_i$) is arbitrarily close to $\|\iota_i(a_i)\|$ (norm in $A$)
  \item and such that \Cref{eq:yxe} holds for $y$ in some neighborhood of $x$ with $a_i$ in place of $a$ (possible by the continuity of $A_i$).
  \end{itemize}
\end{proof}

%%%%%%%%%%%%%%%%%%%%%%%%%%%%%%%%%%%%%%%%%%%%%%%%%%%%%%%%%%%%%%%%%%%%%%%%%%%%%
%%%%%%%%%%%%%%%%%%%%%%%%%%%%%%%%%%%%%%%%%%%%%%%%%%%%%%%%%%%%%%%%%%%%%%%%%%%%%
\section{A comment on the literature}\label{subse:lit}

The present side-note is devoted to an issue I believe is present in the proof of \cite[Theorem 3.7]{bl-fr}. Though apparently the problem is fixable, the proof as-is posed some difficulties (for this reader, at least). The setup is as follows: one considers an arbitrary element $a$ in the dense purely algebraic pushout
\begin{equation*}
  A_1*^{\mathrm{alg}}_{C(X)} A_2\subset A_1*_{C(X)} A_2
\end{equation*}
and seeks to show that
\begin{equation*}
  X\ni x\mapsto \|a_x\|
\end{equation*}
is lower semicontinuous. This is done for separable $C^*$-algebras first, in \cite[Lemma 3.5]{bl-fr}, and then generalized to the present setting by
\begin{itemize}
\item first embedding $a$ into a pushout $D_1*_{C(Y)}D_2$ where $C(Y)\subseteq C(X)$ and $D_i\subseteq A_i$ are separable $C^*$-subalgebras;
\item citing \cite[Theorem 4.2]{ped-psh} to conclude that there is an embedding
  \begin{equation}\label{eq:difam}
    D_1*_{C(Y)}D_2\subseteq A_1*_{C(X)}A_2.
  \end{equation}
\end{itemize}

The problem is with this last step: \cite[Theorem 4.2]{ped-psh} proves inclusions of the form
\begin{equation*}
  D_1*_{C}D_2\subseteq A_1*_{C}A_2
\end{equation*}
given inclusions
\begin{equation*}
  C\subseteq D_i\subseteq A_i,\ i=1,2. 
\end{equation*}
Note that the amalgam $C$ is the same on both sides. On the other hand, given that on the right-hand side of \Cref{eq:difam} one amalgamates over a (generally-speaking) {\it larger} algebra $C(X)\supset C(Y)$, it is not at all obvious that the natural map \Cref{eq:difam} is indeed an inclusion. Indeed, it certainly will not be in full generality:

\begin{example}\label{ex:notinj}
  Consider the case where $D_i=A_i$, but $C(Y)$ is trivial (i.e. $\bC$) whereas $C(X)$ is not. \Cref{eq:difam} is then a {\it sur}jection but not an injection.
\end{example}

Note also that one cannot exhaust $A_i$ and $C(X)$ respectively by separable $D_i$ and $C(Y)$ and naively hope for a continuity permanence property under filtered colimits
\begin{equation*}
  A_1*_{C(X)}A_2 = \varinjlim_{D_i,C(Y)} D_1*_{C(Y)}D_2:
\end{equation*}

\begin{example}\label{ex:notcolim}
  Consider the space $X=\bZ_p$ (the $p$-adic integers), expressed as a limit
  \begin{equation*}
    X=\varprojlim_n (X_n:=\bZ/p^n).
  \end{equation*}
  This affords us a filtered-colimit description
  \begin{equation*}
    C(X) = \varinjlim_n C(X_n),
  \end{equation*}
  but if $A\supset C(X)$ denotes the algebra of {\it all} (possibly discontinuous) bounded functions on $X$, then
  \begin{itemize}
  \item all $C(X_n)\subset A$ are continuous simply because $X_n$ are discrete, while
  \item $C(X)\subset A$ isn't,
  \item despite the fact that the inclusion $C(X)\subset A$ is the filtered colimit of the inclusions 
    \begin{equation*}
      C(X_n)\subset A.
    \end{equation*}
  \end{itemize}
\end{example}

\begin{remark}
  Contrast \Cref{ex:notcolim} with \Cref{th:filt}, where the base space for the fields $A_i$ whose colimit is being considered is fixed. Such problems arise precisely when the base space changes, much as in the initial observation that \Cref{eq:difam} need not be an embedding.
\end{remark}

For all of these reasons, it would seem worthwhile to have a proof of field-continuity permanence under pushouts that is independent of the separable case. We sketch such a proof here.

\begin{theorem}\label{th:psh-cont}
  Let $X$ be a compact Hausdorff space and $A_i$, $i\in I$ a family of unital $C(X)$-algebras. If all $A_i$ are continuous then so is the pushout
  \begin{equation*}
    A:=\Asterisk_{C(X)} A_i.
  \end{equation*}
\end{theorem}
\begin{proof}
  As in the proof of \Cref{th:main}, once we have the statement for {\it pairs} of algebras $A_1$ and $A_2$ the general conclusion follows via \Cref{th:filt} by taking a filtered colimit. We thus focus, for the duration of the proof, on the case of just two algebras $A_i$, $i=1,2$.  
  
  The proof will mimic that of \cite[Theorem 3.2]{chi-rf}. Fix a convergent net \Cref{eq:net} in $X$ and consider a representation $\rho$ of $A=A_1*_{C(X)}A_2$ that
  \begin{itemize}
  \item factors through $A_x$ for some $x\in X$ fixed throughout the proof, and
  \item induces a faithful representation of $A_x$. 
  \end{itemize}
  Then, by \Cref{le:fell} (and as in its proof), the continuity of $A_i$, $i=1,2$ implies that the restrictions $\rho_i$ of $\rho$ to $A_i$ can be Fell-approximated by representations $\rho_{i,\alpha}$ factoring respectively through $(A_{i})_{x_{\alpha}}$, in the sense that
  \begin{equation}\label{eq:ia}
    \lim_{\alpha}\|\rho_i(a_i)\xi-\rho_{i,\alpha}(a_i)\xi\|=0,\ i=1,2
  \end{equation}
  for $a_i\in A_i$. Since $C(X)$ acts via the character $p_{x_{\alpha}}$ in both both $\rho_{i,\alpha}$, one can form the amalgamated free product of the $\rho_{i,\alpha}$, $i=1,2$; by \Cref{eq:ia} that free product will Fell-converge to $\rho$, hence the conclusion by \Cref{le:fell}.
\end{proof}

%%%%%%%%%%%%%%%%%%%%%%%%%%%%%%%%%%%%%%%%%%%%%%%%%%%%%%%%%%%%%%%%%%%%%%%%%%%%%%%%%%%%%%%%%%%%%%%%%%%%%%%%%%%%%%%%%%
%%%%%%%%%%%%%%%%%%%%%%%%%%%%%%%%%%%%%%%%%%%%%%%%%%%%%%%%%%%%%%%%%%%%%%%%%%%%%%%%%%%%%%%%%%%%%%%%%%%%%%%%%%%%%%%%%%

%\bibliography{bib}{}
%\bibliographystyle{plain}
\def\polhk#1{\setbox0=\hbox{#1}{\ooalign{\hidewidth
  \lower1.5ex\hbox{`}\hidewidth\crcr\unhbox0}}}

\addcontentsline{toc}{section}{References}

\Addresses

\end{document}